\documentclass[a4paper,12pt]{article}
\usepackage{amsmath,amsthm,amssymb}
\usepackage{amsmath}
\usepackage{amsfonts}

\advance \textheight by \topskip
\advance \textheight by 1pt
\makeatother

\numberwithin{equation}{section}

\newtheorem{theorem}{Theorem}[section]
\newtheorem{prop}[theorem]{Proposition}
\newtheorem{lem}[theorem]{Lemma}

{\theoremstyle{definition}
\newtheorem{definition}[theorem]{Definition}
\newtheorem{rem}[theorem]{Remark}
}

\def\wt{{\rm wt}\,}

\def\b{\begin{equation}}
\def\e{\end{equation}}

\title{Sigma function associated with a hyperelliptic curve with two points at infinity}
\author{Takanori Ayano\footnote{Osaka Central Advanced Mathematical Institute, Osaka Metropolitan University, Osaka, Japan. \newline \hspace{3ex} Email: ayano@omu.ac.jp} \hspace{1ex} and \hspace{1ex} Victor M. Buchstaber\footnote{Steklov Mathematical Institute of Russian Academy of Sciences, Moscow, Russia. \newline \hspace{3ex} Email: buchstab@mi-ras.ru
\newline \hspace{3ex} Key words: hyperelliptic curve with two points at infinity, sigma function, hyperelliptic function. 
\newline \hspace{3ex} MSC classes: 14H42, 14K25, 14H70, 14H81.}}

\date{}

\begin{document}
\maketitle

\begin{flushright}
\textit{Dedicated to the memory of Alexey Vladimirovich Borisov.}
\end{flushright}

\begin{abstract}
Baker constructed basic meromorphic functions on the Jacobian variety of a hyperelliptic curve with two points at infinity. 
We call them Baker functions. 
The construction is based on the Abel-Jacobi map, which allows us to identify the field of meromorphic functions on the Jacobian variety of the curve with the field of meromorphic functions on the symmetric product of the curve. 
In our previous paper, a solution to the KP equation was constructed in terms of the Baker function. 
This paper is devoted to the properties of the Baker functions. 
In this paper, we construct an entire function whose second logarithmic derivatives are the Baker functions. 
We prove that the power series expansion of the entire function around the origin is determined only by the coefficients of the defining equation of the curve and a branch point of the curve algebraically. 
We also describe the quasi-periodicity of the entire function and express the entire function in terms of the Riemann theta function. 
\end{abstract}

\section{Introduction}

In \cite{Kl1} and \cite{Kl2}, Klein generalized the Weierstrass elliptic sigma function to the multidimensional sigma functions associated with hyperelliptic curves.   
On this problem, Klein published 3 works (1886--1890). 
Pay attention to the papers \cite{B-1898} and \cite{B} by Baker. 
In 1923, a 3-volume collection of Klein's scientific works was published. 
There is no doubt that Klein knew Baker's results. 
However, in this collection Klein emphasized that the theory of hyperelliptic sigma functions was still far from complete. 
Klein and Baker did not discuss the equations of mathematical physics. 
The development of the theory of multidimensional sigma functions in the direction of applications to problems of mathematical physics began with the works of Buchstaber, Enolski, and Leykin (cf. \cite{BEL-97-1}, \cite{BEL-97-2}, \cite{BEL-2012}, \cite{BEL-2018}).  
Over the past 30 years, a number of authors have successfully joined in the development of the classical results of Klein and Baker with applications in mathematical physics.

Throughout the present paper, we denote by $\mathbb{Z}_{\ge0}$, $\mathbb{Z}$, $\mathbb{Q}$, and $\mathbb{C}$ the sets of non-negative integers, integers, rational numbers, and complex numbers, respectively.

For a positive integer $g$, let us consider the polynomial in $X$ 
\[
M(X)=X^{2g+1}+\lambda_2X^{2g}+\lambda_4X^{2g-1}+\cdots+\lambda_{4g}X+\lambda_{4g+2}, \qquad \lambda_i\in\mathbb{C}.
\]
We assume that $M(X)$ has no multiple roots and consider the non-singular hyperelliptic curve of genus $g$
\[C=\Bigl\{(X,Y)\in\mathbb{C}^2 \Bigm|Y^2=M(X)\Bigr\}.\]
We assign weights for $X$, $Y$, and $\lambda_i$ as $\wt(X)=2$, $\wt(Y)=2g+1$, and $\wt(\lambda_i)=i$.  
The equation $Y^2=M(X)$ has the homogeneous weight $4g+2$ with respect to the coefficients $\lambda_i$ and the variables $X,Y$.  
Let $\sigma(u)$ with $u={}^t(u_1,u_3,\dots,u_{2g-1})\in\mathbb{C}^g$ be the sigma function associated with $C$, which is a holomorphic function on $\mathbb{C}^g$ (cf. \cite{BEL-97-1}, \cite{BEL-97-2}, \cite{BEL-2012}). 
The coefficients of the power series expansion of this sigma function are polynomials in $\{\lambda_{2i}\}_{i=1}^{2g+1}$.  
Let $\wp_{i,j}=-\partial_{u_j}\partial_{u_i}\log\sigma$, where $\partial_{u_k}=\partial/\partial u_k$. 
The functions $\wp_{i,j}$ are meromorphic functions on the Jacobian variety of $C$.  
We assign weights for $u_i$ and $\wp_{i,j}$ as $\wt(u_i)=-i$ and $\wt(\wp_{i,j})=i+j$. 
In \cite{BEL-97-1}, a solution to the KdV equation was constructed in terms of $\wp_{1,1}$. 
In \cite{AB}, a solution to the KP equation was constructed in terms of $\wp_{1,1}$. 
In \cite{AB} and \cite{BEL-2000}, a solution to the KP equation was constructed in terms of $\wp_{2g-1,2g-1}$. 
We changed the suffixes of $\wp_{i,j}$ in \cite{BEL-97-1} to use the grading. 
The suffix $g$ in \cite{BEL-97-1} is replaced with $1$ and the suffix $1$ in \cite{BEL-97-1} is replaced with $2g-1$. 
 
For a positive integer $g$, let us consider the polynomial in $x$ 
\[
N(x)=\nu_{0}x^{2g+2}+\nu_{2}x^{2g+1}+\cdots+\nu_{4g+2}x+\nu_{4g+4}, \quad \nu_i\in\mathbb{C},\quad\nu_{0}\neq0.
\]
We assume that $N(x)$ has no multiple roots and consider the non-singular hyperelliptic curve of genus $g$
\[V=\Bigl\{(x,y)\in\mathbb{C}^2 \Bigm| y^2=N(x)\Bigr\}.\]
We assign weights for $x$, $y$, and $\nu_i$ as $\wt(x)=2$, $\wt(y)=2g+2$, and $\wt(\nu_i)=i$.  
The equation $y^2=N(x)$ has the homogeneous weight $4g+4$ with respect to the coefficients $\nu_i$ and the variables $x,y$. 
We take $a\in\mathbb{C}$ such that $N(a)=0$. 
In \cite{B}, Baker introduced basic meromorphic functions $\mathcal{P}_{i,j}(v)$ with $v={}^t(v_{2g},v_{2g-2},\dots,v_2)\in\mathbb{C}^g$ and $i,j=2,4,\dots,2g$ on the Jacobian variety of $V$. 
We assign weights for $v_i$ and $\mathcal{P}_{i,j}$ as $\wt(v_i)=-i$ and $\wt(\mathcal{P}_{i,j})=i+j$.  
The functions $\mathcal{P}_{i,j}$ are determined by $\{\nu_{2i}\}_{i=0}^{2g+2}$ and $a$. 
In \cite{B}, Baker used the Abel-Jacobi map and did not introduce a sigma function to define the functions $\mathcal{P}_{i,j}$. 
In \cite{B}, Baker derived a fundamental formula on differential relations between the functions $\mathcal{P}_{i,j}$. 
Further, in \cite{B}, the differential relations between the functions $\mathcal{P}_{i,j}$ were described explicitly for $g=1,2,3$. 
In \cite{M}, in the case $g=3$, it was proved that the function $\mathcal{P}_{2,2}$ satisfies the KP equation. 
In \cite{AB}, we described the differential relations between the functions $\mathcal{P}_{i,j}$ explicitly for any $g\ge1$ and 
proved that the function $\mathcal{P}_{2,2}$ satisfies the KP equation for any $g\ge3$. 
The new results of our paper are as follows. 

\begin{itemize}

\item We describe the relations between $\mathcal{P}_{i,j}$ and $\wp_{k,l}$ explicitly in Proposition \ref{2025.7.24.2}, which is a refinement of \cite[Proposition 7.3 (i)]{AB}.

\item We construct an entire function $H(v)$ such that $\partial_{v_i}\partial_{v_j}\log H(v)=-\mathcal{P}_{i,j}(v)$ for $i,j=2,4,\dots,2g$ in Definition \ref{2025.8.23.8654098} and Theorem \ref{2025.8.23.3}.

\item We prove that the power series expansion of $H(v)$ around the origin is determined only by $a$ and $\{\nu_{2i}\}_{i=0}^{2g+2}$ algebraically in Theorem \ref{2025.8.23.2}.

\item We express the quasi-periodicity of $H(v)$ in Proposition \ref{2025.8.23.4}. 

\item We express $H(v)$ in terms of the Riemann theta function in Proposition \ref{2025.8.23.5}.  
\end{itemize}

The formulas in our paper are consistent with the grading.
Grading is the fundamental difference between the sigma function and the theta function.
The theta function does not allow grading since its arguments are normalized.

In \cite{A-E-E-2003}, \cite{A-E-E-2004}, and \cite{B-1907}, the identities for hyperelliptic functions of genus 2 which are different from the hyperelliptic functions considered in our paper were studied. 

The sigma functions associated with the $(n,s)$ curves were considered in \cite{BEL-99-R}, \cite{BEL-99-2}, \cite{EEL}, and \cite{N1}.  
The sigma functions associated with the telescopic curves were considered in \cite{Aya1} and \cite{Aya2}. 
For the $(n,s)$ curves and the telescopic curves, the coefficients of the power series expansion of the sigma function are polynomials in the coefficients of the defining equations of the curve. 
The sigma functions associated with the Weierstrass curves were considered in \cite{Komeda-Matsutani-Previato}. 
The $(n,s)$ curves, the telescopic curves, and the Weierstrass curves include hyperelliptic curves with one point at infinity. 
On the other hand, they do not include hyperelliptic curves with two points at infinity. 
In \cite{Korotokin} and \cite{Nakayashiki}, the sigma functions associated with general compact Riemann surfaces were defined. 
They are modular invariant, i.e., they do not depend on the choice of a canonical homology basis. 
%The meromorphic 1-forms of the second kind on $V$ described in (\ref{2025.8.12.5}) are different from those considered in \cite{Nakayashiki}. 
The function $H(v)$ considered in our paper is apparently different from the sigma functions defined in \cite{Korotokin} and \cite{Nakayashiki}. 
%In our next paper, we will consider the relations between the sigma functions defined in \cite{Korotokin} and \cite{Nakayashiki} and the function $H(v)$ in detail. 
In the case of general compact Riemann surfaces, we are not talking about a model of a curve in the form of specific algebraic equations. 
The following problem arises. 
Suppose that an algebraic curve is given by an algebraic equation. 
Find a sigma function whose power series expansion uses the coefficients of this algebraic equation. 
In our paper, we solve this problem in the case of hyperelliptic curves with two points at infinity. 
%For our sigma function $H(v)$, we prove not only the modular invariance but also the algebraic property of the power series expansion (see Theorem \ref{2025.8.23.2}). 

The present paper is organized as follows. 
In Section 2, we review the definition and properties of the sigma function associated with a hyperelliptic curve with one point at infinity. 
In Section 3, we review the definition of the hyperelliptic functions associated with a hyperelliptic curve with two points at infinity. 
In Section 4, we construct the sigma function associated with the hyperelliptic curve with two points at infinity and study its properties.

\section{Sigma function associated with a hyperelliptic curve with one point at infinity}\label{2024.10.21.1}

In this section, we review the definition of the sigma function associated with a hyperelliptic curve with one point at infinity and give facts about it which will be used later on.  
For details, see \cite{BEL-97-1}, \cite{BEL-97-2}, and \cite{BEL-2012}.  

For a positive integer $g$, let us consider the polynomial in $X$ 
\[
M(X)=X^{2g+1}+\lambda_2X^{2g}+\lambda_4X^{2g-1}+\cdots+\lambda_{4g}X+\lambda_{4g+2}, \qquad \lambda_i\in\mathbb{C}.
\]
We assume that $M(X)$ has no multiple roots and consider the non-singular hyperelliptic curve of genus $g$
\[C=\Bigl\{(X,Y)\in\mathbb{C}^2 \Bigm|Y^2=M(X)\Bigr\}.\]
We assign weights for $X$, $Y$, and $\lambda_i$ as $\wt(X)=2$, $\wt(Y)=2g+1$, and $\wt(\lambda_i)=i$.  
The equation $Y^2=M(X)$ has the homogeneous weight $4g+2$ with respect to the coefficients $\{\lambda_{2i}\}_{i=1}^{2g+1}$ and the variables $X,Y$.  
A basis of the vector space consisting of holomorphic 1-forms on $C$ is given by
\[
\omega_i=-\frac{X^{g-i}}{2Y}dX, \qquad 1\le i\le g. 
\]
We set $\omega={}^t(\omega_1,\dots,\omega_g)$. 
Let us consider the following meromorphic 1-forms of the second kind on $C$:  
\begin{equation}
\eta_i=-\frac{1}{2Y}\sum_{k=g-i+1}^{g+i-1}(k+i-g)\lambda_{2g+2i-2k-2}X^kdX, \qquad 1\le i\le g,\label{2024.10.21.345}
\end{equation}
which are holomorphic at any point except $\infty$. 
In (\ref{2024.10.21.345}), we set $\lambda_0=1$. 
For example, for $g=1$, we have 
\[\eta_1=-\frac{X}{2Y}dX.\]
%and for $g=2$ we have 
%\[\eta_1=-\frac{X^2}{2Y}dX, \qquad \eta_2=-\frac{\lambda_4X+2\lambda_2X^2+3X^3}{2Y}dX.\]
Let $\{A_i, B_i\}_{i=1}^g$ be a canonical basis in the one-dimensional homology group of the curve $C$. 
We define the period matrices by 
\[
2\omega'=\left(\int_{A_j}\omega_i\right),\quad
2\omega''=\left(\int_{B_j}\omega_i\right),\quad 
-2\eta'=\left(\int_{A_j}\eta_i\right),
\quad
-2\eta''=\left(\int_{B_j}\eta_i\right).
\]
We define the lattice of periods $\Lambda=\bigl\{2\omega'm_1+2\omega''m_2\mid m_1,m_2\in\mathbb{Z}^g\bigr\}$ and consider the Jacobian variety $\operatorname{Jac}(C)=\mathbb{C}^g/\Lambda$. 
The normalized period matrix is given by $\tau=(\omega')^{-1}\omega''$. 
Let $\tau\delta'+\delta''$ with $\delta',\delta''\in\mathbb{R}^g$ be the Riemann constant with respect to $\bigl(\{A_i, B_i\}_{i=1}^g,\infty\bigr)$. We denote the imaginary unit by $\textbf{i}$.
The sigma function $\sigma(u)$ associated with the curve $C$, $u={}^t(u_1, u_3, \dots, u_{2g-1})\in\mathbb{C}^g$, is defined by
\begin{equation}
\sigma(u)=\varepsilon\exp\biggl(\frac{1}{2}{}^tu\eta'(\omega')^{-1}u\biggr)\theta\begin{bmatrix}\delta'\\ \delta'' \end{bmatrix}\bigl((2\omega')^{-1}u,\tau\bigr),\label{2023.8.23.1357}
\end{equation}
where $\theta\begin{bmatrix}\delta'\\ \delta'' \end{bmatrix}(u,\tau)$ is the Riemann theta function with the characteristics $\begin{bmatrix}\delta'\\ \delta'' \end{bmatrix}$ defined by
\[
\theta\begin{bmatrix}\delta'\\ \delta'' \end{bmatrix}(u,\tau)=\sum_{n\in\mathbb{Z}^g}\exp\bigl\{\pi\textbf{i}\,{}^t(n+\delta')\tau(n+\delta')+2\pi\textbf{i}\,{}^t(n+\delta')(u+\delta'')\bigr\}
\]
and $\varepsilon$ is a non-zero constant. 
The characteristics of this sigma function correspond to the vector of the Riemann constant. 
Let $K=\begin{pmatrix}\omega'&\omega''\\\eta'&\eta''\end{pmatrix}$. 

\begin{prop}[{\cite[Lemma 1.1]{BEL-97-1}}, {\cite[p.~191]{N1}}]\label{2025.8.14.1}
We have 
\[{}^tK\begin{pmatrix}O&E_g\\-E_g&O\end{pmatrix}K=-\frac{\pi\textbf{i}}{2}\begin{pmatrix}O&E_g\\-E_g&O\end{pmatrix},\]
where $E_g$ is the identity matrix of size $g$ and $O$ is the $g\times g$ zero matrix. 
\end{prop}

\begin{prop}[{\cite[pp.~7, 8]{BEL-97-1}}]\label{2025.2.23.18765432042224455}
For $m_1,m_2\in\mathbb{Z}^g$ and $u\in\mathbb{C}^g$, we have
\begin{align*}
&\sigma(u+2\omega'm_1+2\omega''m_2)/\sigma(u) \\
&=(-1)^{2({}^t\delta'm_1-{}^t\delta''m_2)+{}^tm_1m_2}\exp\bigl\{{}^t(2\eta'm_1+2\eta''m_2)(u+\omega'm_1+\omega''m_2)\bigr\}.
\end{align*}
\end{prop}

For $n\ge0$, let $p_n(T)$ be the polynomial of $T_1, T_2, \dots$ defined by 
\begin{equation}
\sum_{i=0}^{\infty}\frac{1}{i!}\left(\sum_{j=1}^{\infty}T_jk^j\right)^i=\sum_{n=0}^{\infty}p_n(T)k^n,\label{2022.3.16.1}
\end{equation}
where $k$ is a variable, i.e., $p_n(T)$ is the coefficient of $k^n$ in the left-hand side of (\ref{2022.3.16.1}). 
For example, we have 
\[p_0(T)=1,\;\;\;\;p_1(T)=T_1,\;\;\;\;p_2(T)=T_2+\frac{T_1^2}{2},\;\;\;\;p_3(T)=T_3+T_1T_2+\frac{T_1^3}{6}.\]
For $n<0$, let $p_n(T)=0$. 
Let 
\[S(T)=\det \bigl(p_{g+j+1-2i}(T)\bigr)_{1\le i,j\le g}.\]

\begin{lem}[{\cite[Section 4]{BEL-99-R}}]
The polynomial $S(T)$ is a polynomial in the variables $T_1,T_3,\dots, T_{2g-1}$. 
\end{lem}

%Let ${\boldsymbol \l}=\{\lambda_2,\lambda_4,\dots,\lambda_{4g}, \lambda_{4g+2}\}$. 
Let $S(u)=S(T)|_{T_i=u_i}$. We assign weights for $u_i$ as $\wt(u_i)=-i$.

\begin{theorem}[{\cite[Theorem 6.3]{BEL-99-R}, \cite[Theorem 3]{N1}}]\label{2023.8.23.1}
The sigma function $\sigma(u)$ is a holomorphic function on $\mathbb{C}^g$ and we have the unique constant $\varepsilon$ in $(\ref{2023.8.23.1357})$ such that the power series expansion of $\sigma(u)$ around the origin has the following form$:$  
\begin{equation}
\sigma(u)=S(u)+\sum_{\sum_{i=1}^g(2i-1)n_i>g(g+1)/2}\gamma_{n_1,\dots,n_g}u_1^{n_1}\cdots u_{2g-1}^{n_g},\label{4.27.1}
\end{equation}
where $\gamma_{n_1,\dots,n_g}\in\mathbb{Q}\bigl[\{\lambda_{2i}\}_{i=1}^{2g+1}\bigr]$ and the right-hand side of $(\ref{4.27.1})$ is homogeneous
of degree $-g(g+1)/2$ with respect to $\{\lambda_{2i}\}_{i=1}^{2g+1}$ and $\{u_{2i-1}\}_{i=1}^g$.
\end{theorem}

We take the constant $\varepsilon$ in (\ref{2023.8.23.1357}) such that the expansion (\ref{4.27.1}) holds. 
Then the sigma function $\sigma(u)$ does not depend on the choice of a canonical basis $\{A_i, B_i\}_{i=1}^g$ in the one-dimensional homology group of the curve $C$ and is determined only 
by the coefficients $\{\lambda_{2i}\}_{i=1}^{2g+1}$ of the defining equation of the curve $C$. 
Let $\wp_{i,j}=-\partial_{u_j}\partial_{u_i}\log\sigma$, where $\partial_{u_k}=\partial/\partial u_k$. 
We assign weights for $\wp_{i,j}$ as $\wt(\wp_{i,j})=i+j$. 
%For an integer $k\ge2$, let $\wp_{i_1,\dots,i_k}=-\partial_{u_{i_1}}\cdots\;\partial_{u_{i_k}}\log\sigma$, where $\partial_{u_l}=\partial/\partial u_l$. 
%We assign weights for $\wp_{i_1,\dots,i_k}$ as $\wt(\wp_{i_1,\dots,i_k})=i_1+\cdots+i_k$. 

\begin{rem}
In \cite[Theorem 4.12]{BEL-97-1}, it was proved that the function 
\[\mathcal{G}(t_1,t_3,\dots,t_{2g-1})=2\wp_{1,1}(t_1,t_3,\dots,t_{2g-1})+2\lambda_2/3\]
satisfies the KdV equation
\[4\partial_{t_3}\mathcal{G}+6\mathcal{G}\partial_{t_1}\mathcal{G}-\partial_{t_1}^3\mathcal{G}=0.\]
\end{rem}

\begin{rem}
We consider the case $g\ge2$ and for $g\ge3$ take constants $\varrho_i\in\mathbb{C}$ with $3\le i\le g$. 
Let us consider the function
\[\Upsilon(t_1,t_2,t_3)=-2\wp_{1,1}(t_1+2\sqrt{\lambda_2}t_2, -4t_3, \varrho_3,\dots,\varrho_g).\]
In \cite[Proposition 3.9]{AB}, it was proved that the function $\Upsilon$ satisfies the KP equation
\[\partial_{t_1}(\partial_{t_3}\Upsilon+6\Upsilon\partial_{t_1}\Upsilon+\partial_{t_1}^3\Upsilon)=\partial_{t_2}^2\Upsilon.\]

\end{rem}

\begin{rem}
In \cite[p.~170]{BEL-2000}, it was pointed out that if $g\ge3$, under certain restrictions on the coefficients of the defining equation of the curve, $\wp_{2g-1,2g-1}$ is a solution to the KP equation. 
We consider the case $g\ge3$, assume $\lambda_{4g+2}\neq0$, and for $g\ge4$ take constants $b_i\in\mathbb{C}$ with $1\le i\le g-3$. 
Let 
\[\varphi(t_1,t_2,t_3)=-2\wp_{2g-1,2g-1}\left(b_1,\dots,b_{g-3}, \mathfrak{c}t_3, \mathfrak{d}t_2, t_1+\mathfrak{e}t_2\right)-\mathfrak{f},\]
where 
\[\mathfrak{c}=-16\lambda_{4g+2},\quad \mathfrak{d}=2\sqrt{-3\lambda_{4g+2}},\quad \mathfrak{e}=\frac{\lambda_{4g}}{\sqrt{-3\lambda_{4g+2}}},\quad \mathfrak{f}=\frac{2}{3}\lambda_{4g-2}+\frac{\lambda_{4g}^2}{18\lambda_{4g+2}}.\]
In \cite[Corollary 3.12]{AB}, it was proved that if $\lambda_{4g+2}\neq0$, the function $\varphi$ satisfies the KP equation
\[\partial_{t_1}(\partial_{t_3}\varphi+6\varphi\partial_{t_1}\varphi+\partial_{t_1}^3\varphi)=\partial_{t_2}^2\varphi.\]

\end{rem}

\section{Hyperelliptic functions associated with a hyperelliptic curve with two points at infinity}\label{2024.10.25.1}

In this section, we define basic meromorphic functions on the Jacobian variety of a hyperelliptic curve with two points at infinity in accordance with \cite[p.~145]{B}. 
For details, see \cite[Section 4]{AB}. 

For a positive integer $g$, let us consider the polynomial in $x$ 
\[
N(x)=\nu_{0}x^{2g+2}+\nu_{2}x^{2g+1}+\cdots+\nu_{4g+2}x+\nu_{4g+4}, \quad \nu_i\in\mathbb{C},\quad\nu_{0}\neq0.
\]
We assume that $N(x)$ has no multiple roots and consider the non-singular hyperelliptic curve of genus $g$
\[V=\Bigl\{(x,y)\in\mathbb{C}^2 \Bigm| y^2=N(x)\Bigr\}.\]
We assign weights for $x$, $y$, and $\nu_i$ as $\wt(x)=2$, $\wt(y)=2g+2$, and $\wt(\nu_i)=i$.  
The equation $y^2=N(x)$ has the homogeneous weight $4g+4$ with respect to the coefficients $\{\nu_{2i}\}_{i=0}^{2g+2}$ and the variables $x,y$.
A basis of the vector space consisting of holomorphic 1-forms on $V$ is given by 
\[
\mu_i=\frac{x^{i-1}}{2y}dx, \qquad 1\le i\le g. 
\]
We set $\mu={}^t(\mu_1,\dots,\mu_g)$. 
Let $\{\mathfrak{a}_i,\mathfrak{b}_i\}_{i=1}^g$ be a canonical basis in the one-dimensional homology group of the curve $V$. 
We define the period matrices by 
\[
2\mu'=\left(\int_{\mathfrak{a}_j}\mu_i\right),\qquad
2\mu''=\left(\int_{\mathfrak{b}_j}\mu_i\right). 
\]
We define the lattice of periods $L=\bigl\{2\mu'm_1+2\mu''m_2\mid m_1,m_2\in\mathbb{Z}^g\bigr\}$ and consider the Jacobian variety $\operatorname{Jac}(V)=\mathbb{C}^g/L$. 
%We call a meromorphic function on $\operatorname{Jac}(V)$ a \textit{$(2,2g+2)$ function}. 
We take $a\in\mathbb{C}$ such that $N(a)=0$. 
Let $\operatorname{Sym}^g(V)$ be the $g$-th symmetric product of $V$.   
Let $\mathcal{F}\bigl(\operatorname{Sym}^g(V)\bigr)$ and $\mathcal{F}\bigl(\operatorname{Jac}(V)\bigr)$ be the fields of meromorphic functions on $\operatorname{Sym}^g(V)$ and $\operatorname{Jac}(V)$, respectively. 
Let us consider the Abel-Jacobi map
\[I: \quad\operatorname{Sym}^g(V)\to\operatorname{Jac}(V),\qquad \sum_{i=1}^gQ_i\mapsto\sum_{i=1}^g\int_{(a,0)}^{Q_i}\mu.\]
The map $I$ induces the isomorphism of fields 
\[I^* :\quad \mathcal{F}\bigl(\operatorname{Jac}(V)\bigr)\to\mathcal{F}\bigl(\operatorname{Sym}^g(V)\bigr),\qquad \phi\mapsto \phi\circ I.\]
For $(x_i,y_i)\in V$ with $1\le i\le g$, let 
\[R(x)=(x-a)(x-x_1)\cdots(x-x_g),\qquad R'(x)=\frac{d}{dx}R(x).\]
For variables $e_1,e_2$, we set 
\begin{gather*}
\nabla=\sum_{i=1}^g\frac{y_i}{(e_1-x_i)(e_2-x_i)R'(x_i)},\quad f(e_1,e_2)=\sum_{i=0}^{g+1}e_1^ie_2^i\bigl\{2\nu_{4g+4-4i}+\nu_{4g+2-4i}(e_1+e_2)\bigr\},
\end{gather*}
where we set $\nu_{-2}=0$. 
We set 
\[F(e_1,e_2)=f(e_1,e_2)R(e_1)R(e_2)+(e_1-e_2)^2R(e_1)^2R(e_2)^2\nabla^2-N(e_1)R(e_2)^2-N(e_2)R(e_1)^2.\]
%We have $F(e_1,e_2)\in\mathbb{C}[e_1,e_2]$. 
Note that $F(e_1,e_2)$ is a symmetric polynomial in $e_1$ and $e_2$. 
The polynomial $F(e_1,e_2)$ can be divided by $(e_1-e_2)^2R(e_1)R(e_2)$ (see \cite[Lemmas 4.3 and 4.4]{AB}). 
Let $G(e_1,e_2)=F(e_1,e_2)/\bigl\{(e_1-e_2)^2R(e_1)R(e_2)\bigr\}$. 
Then $G(e_1,e_2)$ is a symmetric polynomial in $e_1$ and $e_2$ of degree at most $g-1$ in each variable. 
We assign weights for $a$, $x_i$, $y_i$, and $e_i$ as $\wt(a)=\wt(x_i)=\wt(e_i)=2$ and $\wt(y_i)=2g+2$.   
Then $G(e_1,e_2)$ has the homogeneous weight $4g$.  

\begin{definition}[{\cite[p.~145]{B}}]\label{2024.11.29.12345}

\noindent (i) For $1\le i,j\le g$, we define $P_{2g+2-2i, 2g+2-2j}\in\mathcal{F}\bigl(\operatorname{Sym}^g(V)\bigr)$ by 
\[
\sum_{i,j=1}^gP_{2g+2-2i, 2g+2-2j}e_1^{i-1}e_2^{j-1}=G(e_1,e_2).
\]

\noindent (ii) For $i,j=2,4,\dots,2g$, we define the meromorphic functions $\mathcal{P}_{i,j}(v)$ with $v={}^t(v_{2g},v_{2g-2},\dots,v_2)\in\mathbb{C}^g$ on $\operatorname{Jac}(V)$ by $\mathcal{P}_{i,j}=(I^*)^{-1}(P_{i,j})$. 
\end{definition}

%\begin{ex}
For example, for $g=1$, we have 
\[P_{2,2}=\frac{a(\nu_2+2a\nu_0)x_1+\nu_6+2a\nu_4+2a^2\nu_2+2a^3\nu_0}{x_1-a}.\]
%\end{ex}

Since $G(e_1,e_2)$ is a symmetric polynomial in $e_1$ and $e_2$, we have $\mathcal{P}_{i,j}=\mathcal{P}_{j,i}$ for any $i,j$. 
We assign weights for $v_i$ and $\mathcal{P}_{i,j}$ as $\wt(v_i)=-i$ and $\wt(\mathcal{P}_{i,j})=i+j$.  
%Let $\mathcal{P}_{i,j,k_1,\dots,k_l}=\partial_{v_{k_1}}\cdots\;\partial_{v_{k_l}}\mathcal{P}_{i,j}$.  
%We assign weights for $v_i$ and $\mathcal{P}_{i_1,\dots,i_k}$ as $\wt(v_i)=-i$ and $\wt(\mathcal{P}_{i_1,\dots,i_k})=i_1+\cdots+i_k$. 

\begin{rem}
We consider the case $g\ge3$ and for $g\ge4$ take constants $c_i\in\mathbb{C}$ with $1\le i\le g-3$. 
Let 
\[\psi(t_1,t_2,t_3)=-2\mathcal{P}_{2,2}\left(c_1,\dots,c_{g-3},\alpha t_3, \beta t_2, t_1+\gamma t_2\right)-\delta,\]
where 
\[\alpha=-16\nu_{0},\quad \beta=2\sqrt{-3\nu_{0}},\quad \gamma=\frac{\nu_{2}}{\sqrt{-3\nu_{0}}},\quad \delta=\frac{2}{3}\nu_{4}+\frac{\nu_{2}^2}{18\nu_{0}}.\]
In \cite[Corollary 5.8]{AB}, it was proved that the function $\psi$ satisfies the KP equation
\[\partial_{t_1}(\partial_{t_3}\psi+6\psi\partial_{t_1}\psi+\partial_{t_1}^3\psi)=\partial_{t_2}^2\psi.\]
\end{rem}

\section{Sigma function associated with the hyperelliptic curve with two points at infinity}

Let us express $N(x)$ in the following form: 
\[N(x)=\nu_{0}(x-a)\prod_{i=1}^{2g+1}(x-a_i),\qquad a_i\in\mathbb{C}.\]
We take $\mathfrak{s}, \mathfrak{t}\in\mathbb{C}$ such that $\mathfrak{s}\mathfrak{t}\neq0$ and $\mathfrak{s}^{2g+1}/\mathfrak{t}^2=N'(a)$. 
%Let $\mathfrak{s}=N'(a)^{2/(2g+1)}$ and $\mathfrak{t}=N'(a)^{1/2}$.  
We assign weights for $\mathfrak{s}$ and $\mathfrak{t}$ as $\wt(\mathfrak{s})=4$ and $\wt(\mathfrak{t})=2g+1$. 
Let us consider the polynomial 
\[
\widetilde{M}(X)=\prod_{i=1}^{2g+1}\left(X-\frac{\mathfrak{s}}{a_i-a}\right)
\]
and the hyperelliptic curve $\widetilde{C}$ of genus $g$ defined by 
\[\widetilde{C}=\Bigl\{(X,Y)\in\mathbb{C}^2 \Bigm| Y^2=\widetilde{M}(X)\Bigr\}.\]
We have the following isomorphism from $V$ to $\widetilde{C}$$:$
\[\zeta\colon\quad V\to \widetilde{C},\qquad (x,y)\mapsto (X,Y)=\left(\frac{\mathfrak{s}}{x-a}, \frac{\mathfrak{t}\;y}{(x-a)^{g+1}}\right)\]
(cf. \cite[pp.~144, 145]{B}). 
Let $D$ be the $g\times g$ regular matrix defined by 
\begin{equation}
{}^t\bigl(\zeta^*(\omega_1),\dots,\zeta^*(\omega_g)\bigr)=D\mu,\label{2025.8.12.4}
\end{equation}
where $\zeta^*(\omega_i)$ is the pullback of the holomorphic 1-form $\omega_i$ on $\widetilde{C}$ with respect to the map $\zeta$. 
For $n,k\in\mathbb{Z}_{\ge0}$, we denote the binomial coefficient by $\begin{pmatrix}n\\k\end{pmatrix}$. 
%For $k>n$, we set $\begin{pmatrix}n\\k\end{pmatrix}=0$. 

\begin{prop}\label{2025.7.24.1}
For $1\le i,j\le g$, the $(i,j)$ element of $D$ is $\mathfrak{t}^{-1}\mathfrak{s}^{g+1-i}\begin{pmatrix}i-1\\j-1\end{pmatrix}(-a)^{i-j}$. 
\end{prop}

\begin{proof}
From
\[\zeta^*(\omega_i)=\mathfrak{t}^{-1}\mathfrak{s}^{g+1-i}(x-a)^{i-1}\frac{dx}{2y}=\mathfrak{t}^{-1}\mathfrak{s}^{g+1-i}\sum_{j=1}^i\begin{pmatrix}i-1\\j-1\end{pmatrix}(-a)^{i-j}\mu_j,\]
we obtain the statement of the proposition. 
\end{proof}

For $0\le i\le 2g+1$, we define $\widetilde{\lambda}_{2i}\in\mathbb{C}$ by the following equality:
\[\widetilde{M}(X)=\widetilde{\lambda}_0X^{2g+1}+\widetilde{\lambda}_2X^{2g}+\widetilde{\lambda}_4X^{2g-1}+\cdots+\widetilde{\lambda}_{4g}X+\widetilde{\lambda}_{4g+2}.\]

For a positive integer $k$, let $N^{(k)}(x)=(d^k/dx^k)N(x)$. 

\begin{lem}\label{2025.8.10.000987}
For $0\le i\le 2g+1$, we have 
\[\widetilde{\lambda}_{2i}=\mathfrak{s}^i\frac{N^{(i+1)}(a)}{(i+1)!N'(a)}.\]
\end{lem}

\begin{proof}
For $1\le i\le 2g+1$, let $\widetilde{a}_i=a-a_i$. 
For $1\le k\le 2g+1$, let $h_k$ be the elementary symmetric polynomial of degree $k$ in $\widetilde{a}_1,\dots,\widetilde{a}_{2g+1}$. 
We set $h_0=1$. 
For $1\le k\le 2g+2$, we have $N^{(k)}(a)=\nu_0k!h_{2g+2-k}$. 
For $0\le i\le 2g+1$, we have 
\[\widetilde{\lambda}_{2i}=\mathfrak{s}^i\frac{h_{2g+1-i}}{h_{2g+1}}=\mathfrak{s}^i\frac{N^{(i+1)}(a)}{(i+1)!N'(a)}.\]

\end{proof}

\begin{rem}\label{2025.8.11.87654}
For $0\le i\le 2g+1$, the coefficient $\widetilde{\lambda}_{2i}$ has the homogeneous weight $2i$ with respect to $\mathfrak{s}$, $a$, and $\{\nu_{2i}\}_{i=0}^{2g+2}$. 
\end{rem}

Let 
\begin{align}
\begin{split}
\widetilde{f}(e_1,e_2)&=\sum_{i=0}^{g}e_1^ie_2^i\left\{2\widetilde{\lambda}_{4g+2-4i}+\widetilde{\lambda}_{4g-4i}(e_1+e_2)\right\},\\ 
\overline{f}(e_1,e_2)&=\mathfrak{t}^{-2}(e_1-a)^{g+1}(e_2-a)^{g+1}\widetilde{f}\left(\frac{\mathfrak{s}}{e_1-a}, \frac{\mathfrak{s}}{e_2-a}\right).\label{2025.7.25.111}
\end{split}
\end{align}
The polynomial $\overline{f}(e_1,e_2)$ is a symmetric polynomial in $e_1$ and $e_2$. 
From $\widetilde{\lambda}_{4g+2}\neq0$, the degree of $\overline{f}(e_1,e_2)$ is $g+1$ in each variable. 
There exist complex numbers $\left\{\mathfrak{n}_{i,j}\right\}_{i,j=1}^g$ such that $\mathfrak{n}_{i,j}=\mathfrak{n}_{j,i}$ and
\begin{equation}
\overline{f}(e_1,e_2)=f(e_1,e_2)+(e_1-e_2)^2\sum_{i,j=1}^g\mathfrak{n}_{i,j}e_1^{i-1}e_2^{j-1}\label{2025.7.25.1689}
\end{equation}
(see the proof of Proposition 7.3 in \cite{AB}). 
Let 
\[\overline{f}(e_1,e_2)=\sum_{i,j=1}^{g+2}\widetilde{\mathfrak{n}}_{i,j}e_1^{i-1}e_2^{j-1},\qquad \widetilde{\mathfrak{n}}_{i,j}\in\mathbb{C}.\]

\begin{lem}\label{2025.7.26.123}
For $1\le j\le i\le g$, we have 
\begin{align*}
\widetilde{\mathfrak{n}}_{i+2,j}=\mathfrak{t}^{-2}\Biggl\{&\sum_{k=0}^{g-i}2\widetilde{\lambda}_{4g+2-4k}\mathfrak{s}^{2k}\begin{pmatrix}g+1-k\\i+1\end{pmatrix}\begin{pmatrix}g+1-k\\j-1\end{pmatrix}(-a)^{2g+2-2k-i-j}\\
&+\sum_{k=0}^{g-i-1}\widetilde{\lambda}_{4g-4k}\mathfrak{s}^{2k+1}\begin{pmatrix}g-k\\i+1\end{pmatrix}\begin{pmatrix}g+1-k\\j-1\end{pmatrix}(-a)^{2g+1-2k-i-j}\\
&+\sum_{k=0}^{g-i}\widetilde{\lambda}_{4g-4k}\mathfrak{s}^{2k+1}\begin{pmatrix}g+1-k\\i+1\end{pmatrix}\begin{pmatrix}g-k\\j-1\end{pmatrix}(-a)^{2g+1-2k-i-j}\Biggr\}. 
\end{align*}
\end{lem}

\begin{proof}
From (\ref{2025.7.25.111}), we have 
\begin{align*}
\overline{f}(e_1,e_2)=\mathfrak{t}^{-2}\sum_{k=0}^g\Biggl\{&2\widetilde{\lambda}_{4g+2-4k}\mathfrak{s}^{2k}(e_1-a)^{g+1-k}(e_2-a)^{g+1-k}\\
&+\widetilde{\lambda}_{4g-4k}\mathfrak{s}^{2k+1}(e_1-a)^{g-k}(e_2-a)^{g+1-k}\\
&+\widetilde{\lambda}_{4g-4k}\mathfrak{s}^{2k+1}(e_1-a)^{g+1-k}(e_2-a)^{g-k}\Biggr\}. 
\end{align*}
By the direct calculations, we obtain the statement of the lemma. 
\end{proof}

If $i\notin\{1,\dots,g\}$ or $j\notin\{1,\dots,g\}$, then we set $\mathfrak{n}_{i,j}=0$. 

\begin{prop}\label{2025.7.25.1945872}
%$(\mathrm{i})$ For $1\le i\le g$, we have $\mathfrak{n}_{i,1}=\widetilde{\mathfrak{n}}_{i+2,1}$. 

%\vspace{1ex}

%\noindent$(\mathrm{ii})$ For $2\le i\le g$, we have $\mathfrak{n}_{i,2}=2\mathfrak{n}_{i+1,1}+\widetilde{\mathfrak{n}}_{i+2,2}$. 

%\vspace{1ex}

For $1\le j\le i\le g$, we have $\mathfrak{n}_{i,j}=2\mathfrak{n}_{i+1,j-1}-\mathfrak{n}_{i+2,j-2}+\widetilde{\mathfrak{n}}_{i+2,j}$. 

\end{prop}

\begin{proof}
%By comparing the coefficients of $e_1^{i+1}$ in (\ref{2025.7.25.1689}), we obtain the statement of (i). 
%By comparing the coefficients of $e_1^{i+1}e_2$ in (\ref{2025.7.25.1689}), we obtain the statement of (ii). 
By comparing the coefficients of $e_1^{i+1}e_2^{j-1}$ in (\ref{2025.7.25.1689}), we obtain the statement of the proposition. 
\end{proof}

\begin{rem} 
From Lemma \ref{2025.7.26.123}, Proposition \ref{2025.7.25.1945872}, and $\mathfrak{n}_{i,j}=\mathfrak{n}_{j,i}$, we can determine any $\mathfrak{n}_{i,j}$ explicitly. 
\end{rem}

\begin{lem}\label{2025.8.11.540812888765232}
We have $\mathfrak{n}_{i,j}\in\mathbb{Q}\bigl[a, \{\nu_{2i}\}_{i=0}^{2g+2}\bigr]$ and 
$\mathfrak{n}_{i,j}$ has the homogeneous weight $4g+4-2i-2j$ with respect to $a$ and $\{\nu_{2i}\}_{i=0}^{2g+2}$. 
\end{lem}

\begin{proof}
From $\mathfrak{s}^{2g+1}/\mathfrak{t}^2=N'(a)$, Lemma \ref{2025.8.10.000987}, and Lemma \ref{2025.7.26.123}, for $1\le j\le i\le g$, we have $\widetilde{\mathfrak{n}}_{i+2,j}\in\mathbb{Q}\bigl[a, \{\nu_{2i}\}_{i=0}^{2g+2}\bigr]$ and 
$\widetilde{\mathfrak{n}}_{i+2,j}$ has the homogeneous weight $4g+4-2i-2j$ with respect to $a$ and $\{\nu_{2i}\}_{i=0}^{2g+2}$. 
%We will prove that the statement of the lemma holds for any $\mathfrak{n}_{i,j}$ with $1\le j\le i\le g$ by induction with respect to $j$. 
From Proposition \ref{2025.7.25.1945872}, for $1\le i\le g$, we have $\mathfrak{n}_{i,1}=\widetilde{\mathfrak{n}}_{i+2,1}$. 
Thus, the statement of the lemma holds for $\mathfrak{n}_{i,1}$ with $1\le i\le g$. 
We take an integer $j_0$ such that $2\le j_0\le g$ and assume that the statement of the lemma holds for any $\mathfrak{n}_{i,j}$ with $1\le j<j_0$ and $j\le i\le g$. 
From Proposition \ref{2025.7.25.1945872}, the statement of the lemma holds for $\mathfrak{n}_{i,j_0}$ with $j_0\le i\le g$. 
By induction, the statement of the lemma holds for any $\mathfrak{n}_{i,j}$ with $1\le j\le i\le g$. 
From $\mathfrak{n}_{i,j}=\mathfrak{n}_{j,i}$, the statement of the lemma holds for any $\mathfrak{n}_{i,j}$ with $1\le i,j\le g$. 
\end{proof}

%\begin{ex}
For example, for $g=1$, we have 
\[\mathfrak{n}_{1,1}=\mathfrak{t}^{-2}(-a\mathfrak{s}\widetilde{\lambda}_4+2a^2\widetilde{\lambda}_6)=-\frac{aN^{(3)}(a)}{6}+\frac{a^2N^{(4)}(a)}{12}=-2a^2\nu_0-a\nu_2.\] 
%For $g=2$, we have 
%\begin{gather*}
%\mathfrak{n}_{1,1}=\mathfrak{t}^{-2}(-a\mathfrak{s}^3\widetilde{\lambda}_4+2a^2\mathfrak{s}^2\widetilde{\lambda}_6-4a^3\mathfrak{s}\widetilde{\lambda}_8+6a^4\widetilde{\lambda}_{10}),\\
%\mathfrak{n}_{1,2}=\mathfrak{n}_{2,1}=\mathfrak{t}^{-2}(a^2\mathfrak{s}\widetilde{\lambda}_8-2a^3\widetilde{\lambda}_{10}),\qquad \mathfrak{n}_{2,2}=\mathfrak{t}^{-2}(-2a\mathfrak{s}\widetilde{\lambda}_8+6a^2\widetilde{\lambda}_{10}).
%\end{gather*}
%\end{ex}

\begin{prop}\label{2025.7.24.2}
For $1\le i,j\le g$ and $v\in\mathbb{C}^g$, we have 
\begin{align*}
&\mathcal{P}_{2g+2-2i, 2g+2-2j}(v)\\
&=-\mathfrak{n}_{i,j}+\mathfrak{t}^{-2}\sum_{k=i}^g\sum_{l=j}^g\mathfrak{s}^{2g+2-k-l}\begin{pmatrix}k-1\\i-1\end{pmatrix}\begin{pmatrix}l-1\\j-1\end{pmatrix}(-a)^{k+l-i-j}\wp_{2k-1,2l-1}(Dv). 
\end{align*}
\end{prop}

\begin{proof}
In the proof of \cite[Proposition 7.3]{AB}, we proved 
\begin{align*}
&\mathfrak{s}^{-2}\mathfrak{t}^2\sum_{i,j=1}^g\Bigl\{\mathcal{P}_{2g+2-2i, 2g+2-2j}(v)+\mathfrak{n}_{i,j}\Bigr\}e_1^{i-1}e_2^{j-1}\\
%&=(e_1-a)^{g-1}(e_2-a)^{g-1}\sum_{i,j=1}^g\wp_{2g+1-2m, 2g+1-2n}(Dv)\left(\frac{\mathfrak{s}}{e_1-a}\right)^{m-1}\left(\frac{\mathfrak{s}}{e_2-a}\right)^{n-1}\\
%&=\sum_{i,j=1}^g\wp_{2g+1-2i, 2g+1-2j}(Dv)\mathfrak{s}^{i+j-2}(e_1-a)^{g-i}(e_2-a)^{g-j}.
&=\sum_{k,l=1}^g\wp_{2k-1, 2l-1}(Dv)\mathfrak{s}^{2g-k-l}(e_1-a)^{k-1}(e_2-a)^{l-1}.
\end{align*}
By comparing the coefficients of $e_1^{i-1}e_2^{j-1}$ of the equality above, we obtain the statement of the proposition. 
\end{proof}

\begin{rem}
Proposition \ref{2025.7.24.2} is a refinement of \cite[Proposition 7.3 (i)]{AB}. 
For $i=j=g$, the equality in Proposition \ref{2025.7.24.2} was given in \cite[Corollary 7.5]{AB}. 
For $g=2$, the equalities in Proposition \ref{2025.7.24.2} were given in \cite[Example 7.6]{AB}. 
\end{rem}

Let $\Omega$ be the $g\times g$ symmetric matrix defined by $\Omega=(\mathfrak{n}_{i,j})_{1\le i,j\le g}$. 
Note that $g(g+1)/2$ is odd if and only if $g$ is congruent to $1$ or $2$ modulo $4$. 
If $g(g+1)/2$ is odd, we set $\chi=\mathfrak{s}^{(g^2-3g-2)/4}\mathfrak{t}$. 
If $g(g+1)/2$ is even, we set $\chi=\mathfrak{s}^{g(g+1)/4}$. 

\begin{definition}\label{2025.8.23.8654098}
Let $H(v)$ be the holomorphic function on $\mathbb{C}^g$ defined by 
\begin{equation}
H(v)=\chi\exp({}^tv\Omega v)\sigma(Dv),\label{2025.8.12.1}
\end{equation}
%\[H(v)=\kappa\exp\left(\sum_{i,j=1}^g\mathfrak{n}_{i,j}v_{2g+2-2i}v_{2g+2-2j}\right)\sigma(Dv),\]
where $\sigma$ is the sigma function associated with the curve $\widetilde{C}$. 
\end{definition}

\begin{theorem}\label{2025.8.23.3}
\noindent For $1\le i,j\le g$ and $v\in\mathbb{C}^g$, we have 
\[\partial_{v_{2g+2-2i}}\partial_{v_{2g+2-2j}}\log H(v)=-\mathcal{P}_{2g+2-2i, 2g+2-2j}(v).\]
\end{theorem}

\begin{proof}
From Propositions \ref{2025.7.24.1} and \ref{2025.7.24.2}, we obtain the statement of the theorem. 
\end{proof} 

Let 
\[\mathcal{R}=\Bigl\{N'(a)^{-m}r \Bigm| m\in\mathbb{Z}_{\ge0}, \;r\in\mathbb{Q}\bigl[a, \{\nu_{2i}\}_{i=0}^{2g+2}\bigr]\Bigr\}.\]

\begin{theorem}\label{2025.8.23.2}
The power series expansion of $H(v)$ around the origin has the following form$:$  
\begin{equation}
H(v)=\sum_{n_1,\dots,n_g\ge0}\xi_{n_1,\dots,n_g}v_{2g}^{n_1}\cdots v_2^{n_g},\label{2025.8.9.1}
\end{equation}
where $\xi_{n_1,\dots,n_g}\in \mathcal{R}$. 
If $g(g+1)/2$ is odd, then the right-hand side of $(\ref{2025.8.9.1})$ is homogeneous
of degree $(g^2-3g-2)/2$ with respect to $a$, $\{\nu_{2i}\}_{i=0}^{2g+2}$, and $\{v_{2i}\}_{i=1}^g$. 
If $g(g+1)/2$ is even, then the right-hand side of $(\ref{2025.8.9.1})$ is homogeneous
of degree $g(g+1)/2$ with respect to $a$, $\{\nu_{2i}\}_{i=0}^{2g+2}$, and $\{v_{2i}\}_{i=1}^g$. 
In particular, the function $H(v)$ does not depend on $\mathfrak{s}$ and $\mathfrak{t}$. 
\end{theorem}

\begin{proof}
From Theorem \ref{2023.8.23.1}, the power series expansion of $\sigma(u)$ around the origin has the following form$:$  
\begin{equation}
\sigma(u)=\sum_{\sum_{i=1}^g(2i-1)n_i\ge g(g+1)/2}\gamma_{n_1,\dots,n_g}u_1^{n_1}\cdots u_{2g-1}^{n_g},\label{2025.8.10.1}
\end{equation}
where $\gamma_{n_1,\dots,n_g}\in\mathbb{Q}\bigl[\{\widetilde{\lambda}_{2i}\}_{i=1}^{2g+1}\bigr]$ and the right-hand side of $(\ref{2025.8.10.1})$ is homogeneous
of degree $-g(g+1)/2$ with respect to $\{\widetilde{\lambda}_{2i}\}_{i=1}^{2g+1}$ and $\{u_{2i-1}\}_{i=1}^g$. 
%From Proposition \ref{2025.7.24.1}, Lemma \ref{2025.8.10.000987}, and (\ref{2025.8.10.1}), we have the power series expansion
%\begin{equation}
%\sigma(Dv)=S(v)+\sum_{\sum_{i=1}^g(2i-1)n_i>g(g+1)/2}\widetilde{\xi}_{n_1,\dots,n_g}v_{2g}^{n_1}\cdots v_2^{n_g},\label{2025.8.10.753987}
%\end{equation}
%where $\widetilde{\xi}_{n_1,\dots,n_g}\in\mathbb{Q}\bigl(a, \{\nu_{2i}\}_{i=0}^{2g+2}, \mathfrak{s}, \mathfrak{t}\bigr)$ and the right-hand side of $(\ref{2025.8.10.753987})$ is homogeneous
%of degree $-g(g+1)/2$ with respect to $a$, $\{\nu_{2i}\}_{i=0}^{2g+2}$, $\mathfrak{s}$, and $\mathfrak{t}$. 
If $\gamma_{n_1,\dots,n_g}\neq0$, there exists a positive integer $k$ such that $\wt(\gamma_{n_1,\dots,n_g})=2k$. 
From Lemma \ref{2025.8.10.000987}, there exists $\mathfrak{u}\in \mathcal{R}$ such that $\gamma_{n_1,\dots,n_g}=\mathfrak{s}^k\mathfrak{u}$. 
From (\ref{2025.8.10.1}), we have 
\begin{equation}
2k-\sum_{i=1}^g(2i-1)n_i=-\frac{g(g+1)}{2}.\label{2025.8.9.1345}
\end{equation}
First, we consider the case where $g(g+1)/2$ is odd. 
There exists a non-negative integer $l$ such that $\sum_{i=1}^gn_i=2l+1$. 
From (\ref{2025.8.9.1345}), we have $\sum_{i=1}^gin_i=k+l+(g^2+g+2)/4$.  
By substituting 
\begin{equation}
u_{2i-1}=\mathfrak{t}^{-1}\mathfrak{s}^{g+1-i}\sum_{j=1}^i\begin{pmatrix}i-1\\j-1\end{pmatrix}(-a)^{i-j}v_{2g+2-2j}\label{2025.8.10.567}
\end{equation}
for any $1\le i\le g$ (see Proposition \ref{2025.7.24.1}), we have 
\[
\gamma_{n_1,\dots,n_g}u_1^{n_1}\cdots u_{2g-1}^{n_g}=\mathfrak{s}^k\frac{\mathfrak{s}^{(2l+1)(g+1)-k-l-(g^2+g+2)/4}}{\mathfrak{t}^{2l+1}}\mathfrak{v}\\
=\frac{N'(a)^l\mathfrak{v}}{\mathfrak{s}^{(g^2-3g-2)/4}\mathfrak{t}},
\]
where $\mathfrak{v}\in \mathcal{R}[v_{2g},\dots,v_2]$. 
Here, we used $\mathfrak{s}^{2g+1}/\mathfrak{t}^2=N'(a)$. 
From (\ref{2025.8.10.1}), we have the power series expansion
\begin{equation}
\mathfrak{s}^{(g^2-3g-2)/4}\mathfrak{t}\;\sigma(Dv)=\sum_{n_1,\dots,n_g\ge0}\widetilde{\xi}_{n_1,\dots,n_g}v_{2g}^{n_1}\cdots v_2^{n_g},\label{2025.8.10.497293087}
\end{equation}
where $\widetilde{\xi}_{n_1,\dots,n_g}\in \mathcal{R}$ and the right-hand side of $(\ref{2025.8.10.497293087})$ is homogeneous
of degree $(g^2-3g-2)/2$ with respect to $a$, $\{\nu_{2i}\}_{i=0}^{2g+2}$, and $\{v_{2i}\}_{i=1}^g$.  
Next, we consider the case where $g(g+1)/2$ is even. 
There exists a non-negative integer $m$ such that $\sum_{i=1}^gn_i=2m$. 
From (\ref{2025.8.9.1345}), we have $\sum_{i=1}^gin_i=k+m+g(g+1)/4$.   
By substituting (\ref{2025.8.10.567}) for any $1\le i\le g$, we have 
\[
\gamma_{n_1,\dots,n_g}u_1^{n_1}\cdots u_{2g-1}^{n_g}=\mathfrak{s}^k\frac{\mathfrak{s}^{2m(g+1)-k-m-g(g+1)/4}}{\mathfrak{t}^{2m}}\mathfrak{w}=\frac{N'(a)^m\mathfrak{w}}{\mathfrak{s}^{g(g+1)/4}},
\]
where $\mathfrak{w}\in\mathcal{R}[v_{2g},\dots,v_2]$. 
From (\ref{2025.8.10.1}), we have the power series expansion
\begin{equation}
\mathfrak{s}^{g(g+1)/4}\sigma(Dv)=\sum_{n_1,\dots,n_g\ge0}\overline{\xi}_{n_1,\dots,n_g}v_{2g}^{n_1}\cdots v_2^{n_g},\label{2025.8.11.497293087}
\end{equation}
where $\overline{\xi}_{n_1,\dots,n_g}\in\mathcal{R}$ and the right-hand side of $(\ref{2025.8.11.497293087})$ is homogeneous
of degree $g(g+1)/2$ with respect to $a$, $\{\nu_{2i}\}_{i=0}^{2g+2}$, and $\{v_{2i}\}_{i=1}^g$.  
From Lemma \ref{2025.8.11.540812888765232}, the power series expansion of $\exp({}^tv\Omega v)$ around the origin has the homogeneous weight $0$ with respect to $a$, $\{\nu_{2i}\}_{i=0}^{2g+2}$, and $\{v_{2i}\}_{i=1}^g$. 
Therefore, we obtain the statement of the theorem. 
\end{proof}

Let $\kappa_1,\dots,\kappa_g$ be the meromorphic 1-forms of the second kind on $V$ defined by 
\begin{equation}
{}^t(\kappa_1,\dots,\kappa_g)={}^t\Bigl\{\bigl(\zeta^*(\eta_1),\dots,\zeta^*(\eta_g)\bigr)D\Bigr\}-2\Omega\mu,\label{2025.8.12.5}
\end{equation}
where $\zeta^*(\eta_i)$ is the pullback of the meromorphic 1-form $\eta_i$ on $\widetilde{C}$ with respect to the map $\zeta$. 
The meromorphic 1-forms $\kappa_1,\dots,\kappa_g$ are holomorphic at any point except $(a,0)$. 

\begin{lem}
For $1\le i\le g$, the meromorphic $1$-form $\kappa_i$ has the following form$:$ 
\[\kappa_i=\frac{r}{(x-a)^g}\frac{dx}{2y},\]
where $r\in\mathbb{Q}\bigl[a, \{\nu_{2i}\}_{i=0}^{2g+2},x\bigr]$. 
In particular, $\kappa_i$ does not depend on $\mathfrak{s}$ and $\mathfrak{t}$. 
\end{lem}

\begin{proof}
From (\ref{2024.10.21.345}) and Lemma \ref{2025.8.10.000987}, for $1\le i\le g$, we have 
\[\zeta^*(\eta_i)=\frac{\mathfrak{s}^{g+i}}{\mathfrak{t}}\sum_{k=g-i+1}^{g+i-1}(k+i-g)\frac{N^{(g+i-k)}(a)}{(g+i-k)!N'(a)}(x-a)^{g-k-1}\frac{dx}{2y}.\]
From $\mathfrak{s}^{2g+1}/\mathfrak{t}^2=N'(a)$, Proposition \ref{2025.7.24.1}, and Lemma \ref{2025.8.11.540812888765232}, we obtain the statement of the lemma. 
\end{proof}

For example, for $g=1$, we have 
\[\kappa_1=\frac{2a(2a\nu_0+\nu_2)x+a^2\nu_2+2a\nu_4+\nu_6}{x-a}\frac{dx}{2y}.\]
We define the period matrices by 
\[
-2\kappa'=\left(\int_{\mathfrak{a}_j}\kappa_i\right), \qquad -2\kappa''=\left(\int_{\mathfrak{b}_j}\kappa_i\right).
\]
Hereafter, let $\omega', \omega'', \eta'$, $\eta''$, and $\tau$ be the period matrices of the curve $\widetilde{C}$, which are defined in Section \ref{2024.10.21.1}, with respect to the canonical basis $\bigl\{\zeta(\mathfrak{a}_i), \zeta(\mathfrak{b}_i)\bigr\}_{i=1}^g$ 
in the one-dimensional homology group of the curve $\widetilde{C}$. 
Let $\tau\delta'+\delta''$ with $\delta',\delta''\in\mathbb{R}^g$ be the Riemann constant of $\widetilde{C}$ with respect to $\bigl(\{\zeta(\mathfrak{a}_i), \zeta(\mathfrak{b}_i)\}_{i=1}^g,\infty\bigr)$, which coincides with 
the Riemann constant of $V$ with respect to $\bigl(\{\mathfrak{a}_i, \mathfrak{b}_i\}_{i=1}^g, (a,0)\bigr)$. 
Let $\varepsilon$ be the constant defined in (\ref{2023.8.23.1357}) with respect to $\bigl(\widetilde{C}, \{\zeta(\mathfrak{a}_i), \zeta(\mathfrak{b}_i)\}_{i=1}^g\bigr)$. 

\begin{lem}\label{2025.8.12.3}
We have 
\begin{gather*}
D\mu'=\omega', \qquad D\mu''=\omega'', \qquad \tau=(\mu')^{-1}\mu'', \\
\kappa'={}^tD\eta'+2\Omega\mu', \qquad \kappa''={}^tD\eta''+2\Omega\mu''.
\end{gather*}
\end{lem}
 
\begin{proof}
From (\ref{2025.8.12.4}), we have $D\mu'=\omega'$ and $D\mu''=\omega''$. 
We have $\tau=(\omega')^{-1}\omega''=(\mu')^{-1}D^{-1}D\mu''=(\mu')^{-1}\mu''$. 
From (\ref{2025.8.12.5}), we have $\kappa'={}^tD\eta'+2\Omega\mu'$ and $\kappa''={}^tD\eta''+2\Omega\mu''$. 
\end{proof}

\begin{prop}\label{2025.8.23.4}
For $m_1,m_2\in\mathbb{Z}^g$ and $v\in\mathbb{C}^g$, we have
\begin{align*}
&H(v+2\mu'm_1+2\mu''m_2)/H(v) \\
&=(-1)^{2({}^t\delta'm_1-{}^t\delta''m_2)+{}^tm_1m_2}\exp\bigl\{{}^t(2\kappa'm_1+2\kappa''m_2)(v+\mu'm_1+\mu''m_2)\bigr\}. 
\end{align*}
\end{prop}

\begin{proof}
From (\ref{2025.8.12.1}), Proposition \ref{2025.2.23.18765432042224455}, and Lemma \ref{2025.8.12.3}, we have 
\begin{align*}
&H(v+2\mu'm_1+2\mu''m_2)\\
&=\chi\exp\bigl\{{}^t(v+2\mu'm_1+2\mu''m_2)\Omega(v+2\mu'm_1+2\mu''m_2)\bigr\}\sigma(Dv+2D\mu'm_1+2D\mu''m_2)\\
&=\chi\exp\bigl\{{}^tv\Omega v+{}^t(2\mu'm_1+2\mu''m_2)2\Omega(v+\mu'm_1+\mu''m_2)\bigr\}\\
&\hspace{2ex}\times (-1)^{2({}^t\delta'm_1-{}^t\delta''m_2)+{}^tm_1m_2}\exp\bigl\{{}^t(2\eta'm_1+2\eta''m_2)(Dv+\omega'm_1+\omega''m_2)\bigr\}\sigma(Dv)\\
&=H(v)\exp\bigl\{{}^t(2\mu'm_1+2\mu''m_2)2\Omega(v+\mu'm_1+\mu''m_2)\bigr\}\\
&\hspace{2ex}\times (-1)^{2({}^t\delta'm_1-{}^t\delta''m_2)+{}^tm_1m_2}\exp\bigl\{{}^t(2\eta'm_1+2\eta''m_2)D(v+\mu'm_1+\mu''m_2)\bigr\}\\
&=(-1)^{2({}^t\delta'm_1-{}^t\delta''m_2)+{}^tm_1m_2}\exp\bigl\{{}^t(2\kappa'm_1+2\kappa''m_2)(v+\mu'm_1+\mu''m_2)\bigr\}H(v).
\end{align*}

\end{proof}

\begin{prop}\label{2025.8.23.5}
We have 
\[H(v)=\chi\varepsilon\exp\biggl(\frac{1}{2}{}^tv\kappa'(\mu')^{-1}v\biggr)\theta\begin{bmatrix}\delta'\\ \delta'' \end{bmatrix}\bigl((2\mu')^{-1}v,\tau\bigr).\]
\end{prop}

\begin{proof}
From (\ref{2023.8.23.1357}), (\ref{2025.8.12.1}), and Lemma \ref{2025.8.12.3}, we have 
\begin{align*}
H(v)&=\chi\varepsilon\exp\biggl\{\frac{1}{2}{}^tv\Bigl({}^tD\eta'(\omega')^{-1}D+2\Omega\Bigr)v\biggr\}\theta\begin{bmatrix}\delta'\\ \delta'' \end{bmatrix}\bigl((2\omega')^{-1}Dv,\tau\bigr)\\
&=\chi\varepsilon\exp\biggl\{\frac{1}{2}{}^tv\Bigl({}^tD\eta'(\mu')^{-1}+2\Omega\Bigr)v\biggr\}\theta\begin{bmatrix}\delta'\\ \delta'' \end{bmatrix}\bigl((2\mu')^{-1}v,\tau\bigr)\\
&=\chi\varepsilon\exp\biggl\{\frac{1}{2}{}^tv\bigl({}^tD\eta'+2\Omega\mu'\bigr)(\mu')^{-1}v\biggr\}\theta\begin{bmatrix}\delta'\\ \delta'' \end{bmatrix}\bigl((2\mu')^{-1}v,\tau\bigr)\\
&=\chi\varepsilon\exp\biggl(\frac{1}{2}{}^tv\kappa'(\mu')^{-1}v\biggr)\theta\begin{bmatrix}\delta'\\ \delta'' \end{bmatrix}\bigl((2\mu')^{-1}v,\tau\bigr).
\end{align*}
\end{proof}

Let $\mathcal{K}=\begin{pmatrix}\mu'&\mu''\\\kappa'&\kappa''\end{pmatrix}$. 

\begin{prop}
We have 
\[{}^t\mathcal{K}\begin{pmatrix}O&E_g\\-E_g&O\end{pmatrix}\mathcal{K}=-\frac{\pi\textbf{i}}{2}\begin{pmatrix}O&E_g\\-E_g&O\end{pmatrix}.\]
\end{prop}

\begin{proof}
From Proposition \ref{2025.8.14.1} and Lemma \ref{2025.8.12.3}, we obtain the statement of the proposition. 
\end{proof}

\section*{Acknowledgments}

%The authors would like to thank the referees for reading our manuscript carefully and giving useful comments. 
The authors are grateful to Atsushi Nakayashiki for drawing our attention to the sigma function associated with a hyperelliptic curve with two points at infinity. 
The work of Takanori Ayano was partly supported by MEXT Promotion of Distinctive Joint Research Center Program JPMXP0723833165 and Osaka Metropolitan University Strategic Research Promotion Project (Development of International Research Hubs).


\begin{thebibliography}{00}

%\bibitem{AB2022}
%T. Ayano, V. M. Buchstaber, Relationships Between Hyperelliptic Functions of Genus 2 and Elliptic Functions, SIGMA \textbf{18} (2022), 010, 30 pages. 

\bibitem{A-E-E-2003}
C. Athorne, J. C. Eilbeck, V. Z. Enolskii, Identities for the classical genus two $\wp$ function, J. Geom. Phys. \textbf{48} (2003), 354--368. 

\bibitem{A-E-E-2004}
C. Athorne, J. C. Eilbeck, V. Z. Enolskii, A SL(2) covariant theory of genus 2 hyperelliptic functions, Math. Proc. Camb. Phil. Soc. \textbf{136} (2004), 269--286. 


\bibitem{Aya1}
T. Ayano, Sigma functions for telescopic curves, Osaka J. Math. \textbf{51} (2014), 459--480.


\bibitem{AB}
T. Ayano, V. M. Buchstaber, Hyperelliptic sigma functions and the Kadomtsev-Petviashvili equation, Physica D: Nonlinear Phenomena \textbf{481} (2025), 134819. 

\bibitem{Aya2}
T. Ayano, A. Nakayashiki, On Addition Formulae for Sigma Functions of Telescopic Curves, SIGMA \textbf{9} (2013), 046, 14 pages. 

\bibitem{B-1898}
H. F. Baker, On the Hyperelliptic Sigma Functions, Amer. J. Math. \textbf{20} (1898), 301--384. 

\bibitem{B}
H. F. Baker, On a system of differential equations leading to periodic functions, Acta Math. \textbf{27} (1903), 135--156. 


\bibitem{B-1907}
H. F. Baker, An introduction to the theory of multiply periodic functions, Cambridge University Press, Cambridge, 1907.

%\bibitem{Buchstaber-Bunkova-2020}
%V. M. Buchstaber, E. Yu. Bunkova, Sigma Functions and Lie Algebras of Schr\"{o}dinger Operators, Funct. Anal. Appl. \textbf{54} (2020), 229--240. 

\bibitem{BEL-97-1}    %%% 1
V. M. Buchstaber, V. Z. Enolskiĭ, D. V. Leĭkin, Hyperelliptic Kleinian Functions and Applications, Solitons, Geometry, and Topology: On the Crossroad, Amer. Math. Soc. Transl. Ser. 2, 179, Amer. Math. Soc., Providence, RI, 1997, 1--33. 

\bibitem{BEL-97-2}
V. M. Buchstaber, V. Z. Enolskii, D. V. Leykin, Kleinian Functions, Hyperelliptic Jacobians and Applications, Rev. Math. Math. Phys. 10, 1997, 3--120. 

\bibitem{BEL-99-R}  %%%  7
V. M. Buchstaber, V. Z. Enolskii, D. V. Leykin, Rational Analogs of Abelian Functions, Funct. Anal. Appl. \textbf{33} (1999), 83--94. 

\bibitem{BEL-2000}
V. M. Buchstaber, V. Z. Enolskii, D. V. Leykin, Uniformization of Jacobi Varieties of Trigonal Curves and Nonlinear Differential Equations, Funct. Anal. Appl. \textbf{34} (2000), 159--171. 

\bibitem{BEL-2012}  %%%  8
V. M. Buchstaber, V. Z. Enolski, D. V. Leykin, Multi-Dimensional Sigma-Functions, arXiv:1208.0990, (2012). 

\bibitem{BEL-2018}
V. M. Buchstaber, V. Z. Enolski, D. V. Leykin, $\sigma$-Functions: Old and New Results, Integrable Systems and Algebraic Geometry, 2, London Math. Soc. Lecture Note Ser. 459, Cambridge University Press, 2020, 175--214. 

%\bibitem{BL-2005}
%V. M. Buchstaber and D. V. Leykin, Addition Laws on Jacobian Varieties of Plane Algebraic Curves, Proc. Steklov Inst. Math., \textbf{251} (2005), 49--120.

\bibitem{BEL-99-2}
V. M. Bukhshtaber, D. V. Leikin, V. Z. Enol'skii, $\sigma$-functions of $(n,s)$-curves, Russ. Math. Surv. \textbf{54} (1999), 628--629. 

\bibitem{EEL}
J. C. Eilbeck, V. Z. Enolskii, D. V. Leykin, On the Kleinian Construction of Abelian Functions of Canonical Algebraic Curves, in proceedings of the Conference SIDE III: 
Symmetries and Integrability of Difference Equations (Sabaudia, 1998), CRM Proc. Lecture Notes, \textbf{25}, 
Amer. Math. Soc., Providence, RI, 2000, 121--138. 

\bibitem{Kl1}
F. Klein, Ueber hyperelliptische Sigmafunctionen, Math. Ann. \textbf{27} (1886), 431--464.

\bibitem{Kl2}
F. Klein, Ueber hyperelliptische Sigmafunctionen, Math. Ann. \textbf{32} (1888), 351--380.

\bibitem{Komeda-Matsutani-Previato}
J. Komeda, S. Matsutani, E. Previato, Algebraic Construction of the Sigma Function for General Weierstrass Curves, Mathematics \textbf{10} (2022), 3010. 

\bibitem{Korotokin}
D. Korotkin, V. Shramchenko, On higher genus Weierstrass sigma-function, Physica D: Nonlinear Phenomena \textbf{241} (2012), 2086--2094. 

%\bibitem{BL-2004-Heat-Equations}
%V. M. Buchstaber, D. V. Leykin, Heat Equations in a Nonholonomic Frame, Funct. Anal. Appl. \textbf{38} (2004), 88--101. 

%\bibitem{Bunkova-Buchstaber-2023}
%E. Yu. Bunkova, V. M. Buchstaber, Explicit Formulas for Differentiation of Hyperelliptic Functions, Math. Notes \textbf{114} (2023), 1151--1162. 

\bibitem{M}
S. Matsutani, Hyperelliptic solutions of KdV and KP equations: re-evaluation of Baker's study on hyperelliptic sigma functions, J. Phys. A: Math. Gen. \textbf{34} (2001), 4721--4732. 

\bibitem{N1}   %%%  13
A. Nakayashiki, On Algebraic Expressions of Sigma Functions for $(n,s)$ Curves, Asian J. Math. \textbf{14} (2010), 175--212. 

\bibitem{Nakayashiki}
A. Nakayashiki, Tau function approach to theta functions, Int. Math. Res. Not. \textbf{2016} (2016), 5202--5248. 

\end{thebibliography}
\end{document}